\documentclass{amsart}

\usepackage{graphicx}
\usepackage{algorithm2e}
\newtheorem{theorem}{Theorem}[section]

\usepackage{amsmath}
\theoremstyle{definition}
\newtheorem{definition}[theorem]{Definition}

\newtheorem{proposition}{Proposition}
\theoremstyle{remark}

\numberwithin{equation}{section}
\newcommand{\lp}{\dashv}
\newcommand{\rp}{\vdash}

\DeclareMathOperator{\wt}{wt}

\begin{document}

\title{A Normal Form for HNN-extension of dialgebras}




\author{Chia Zargeh}
\address{Department of Mathematics and Computer Science, Modern College of Business \& Science, Muscat, Sultanate of Oman}

\email{Chia.Zargeh@mcbs.edu.om}


\date{}


\begin{abstract}
We consider a new version of Composition-Diamond Lemma for dialgebras in order to obtain an explicit Groebner-Shirshov basis for HNN-extension of dialgebras and determine a normal form for that.
\end{abstract}

\maketitle

\specialsection*{Introduction}
 
 Associative di-algebras or dialgebras were introduced by Loday \cite{10} and defined as ${K}$-vector spaces equipped with two associative  ${K}$-linear products $\dashv, \vdash \colon D \times D \to D$, called respectively,  the left product and the right product, which satisfy the associativity laws:
    \[ x \dashv (y \vdash z)= x \dashv (y \dashv z),\]
    \[ (x \dashv y) \vdash z= x \vdash (y \vdash z),\]
    \[ x \vdash (y \dashv z)= (x \vdash y) \dashv z,\]
for all $x,y,z \in D$. Dialgebras are closely connected to the notion of Leibniz algebras in the same way as the associative algebras are connected to Lie algebras. Indeed, Loday in \cite{11} showed that any dialgebra $(D, \dashv, \vdash)$ becomes a Leibniz algebra $D_{Leib}$ under the Leibniz bracket $[x, y] := x \dashv y - y \vdash x$ and the universal enveloping algebra of a Leibniz algebra has the structure of a dialgebra. Interesting properties of dialgebras allow to extend classical results. For instance, Bremner \textit{et al.} in \cite{4} obtained a new variety of nonassociative triple systems and provided a generalized statement of the BSO algorithm called Jordan triple disystems. Also, the concept of digroups as a generalization of continuous groups and dialgebra digroup have been studied by Salazar-Díaz \textit{et al.}~ \cite{12}. Some combinatorial studies of dialgebras can be found in \cite{8}, \cite{18} and \cite{19} as well. 

The concept of HNN-extension is an important construction in combinatorial group theory and it was originally introduced by Higman, Neumann and Neumann in \cite{6} stating that if $A_1$ and $A_2$ are isomorphic subgroups of a group $S$, then it is possible to find a group $H$ containing $S$ such that $A_1$ and $A_2$ are conjugate to each other in $H$ and $S$ is embeddable in $H$. The concept of HNN-extension was constructed for Lie algebras in independent works by Lichtman and Shirvani \cite{9} and Wasserman \cite{16}, and it has recently been spread to generalized versions of Lie algebras, namely, Leibniz algebras, Lie superalgebras, Hom-setting of Lie algebras in \cite{7} and \cite{7}, \cite{8} and \cite{15}, respectively. The HNN-extension of dialgebras was firstly introduced in \cite{8} as an effective approach for the construction of HNN-extensions of Leibniz algebras. Ladra \textit{et al.} in \cite{8} used Shirshov's algorithm \cite{13,14} in order to show that every dialgebra embeds inside its HNN-extension and then proved analogues embedding theorem for the case of Leibniz algebras. In this note, we intend to obtain an explicit Groebner-Shirshov basis for the HNN-extension of dialgberas by using a new version of Composition-Diamond Lemma. 
The Composition-Diamond Lemma (CD-Lemma, for short) is an essential concept in combinatorial algebra and the key ingredient of Groebner-Shirshov bases theory. It is used to solve various problems such as normal form, word problem, extensions and embedding theorems. The theory of Groebner-Shirshov bases is the parallel theory to Groebner bases \cite{5} introduced for ideals of free (commutative, anti-commutative) nonassociative algebras, free Lie algebras and simplicitly free associative algebras by Shirshov (see \cite{1}, \cite{13}), and it has been actively developed to different algebraic structure since two decades ago. The first version of Groebner-Shirshov bases theory for associative dialgebras was introduced by Bokut \textit{et al.}, in \cite{3}, and they gave Groebner-Shirshov bases for the universal enveloping algebra of a Leibniz algebra, the bar extension of a dialgebra, the free product of two dialgebras and Clifford dialgebra. The new version of Composition-Diamond Lemma was introduced by Zhang and Chen in \cite{17} based on an arbitrary \emph{monomial-center ordering}. Zhang and Chen compared their version with the corresponding results in Bokut's paper \cite{3} and showed that it is useful and convenient in calculation of Groebner-Shirshov bases for free dialgebras. Moreover, they provided a method to find normal forms of elements of an arbitrary disemigroup. Interested reader in Groebner-Shirshov bases and applications is encouraged to study the recently published book by Bokut \textit{et al.}, \cite{2}. 
\\The remainder of this note is organized as follows. In the first section, we recall the new version of CD-Lemma. In the second section, we construct HNN-extension of dialgebras and employ CD-Lemma in order to obtain an explicit Groebner-Shirshov basis and normal form for that. 
\section{New CD-Lemma for dialgebras}\label{newversion}
Let $D\langle X \rangle$ be the free dialgebra over a field $K$ generated by a well-ordered set $X$ and $X^{+}$ the free semigroup generated by $X$ without the unit. For any $u = x_1 \dots x_m \dots x_n \in X^{+}$ with $x_i \in X$, a \emph{normal diword} is written as \[[u]_m = x_1 \dots x_{m-1} \Dot{x_m} x_{m+1} \dots x_n = x_1 \vdash \dots \vdash x_{m-1} \vdash x_{m} \dashv x_{m+1} \dashv \dots \dashv x_{n}.\]
Write \[[X^{+}]_{\omega}=\{[u]_{m} \mid u \in X^{+}, m \in \mathbb{Z}^{+}, 1 \leq m \leq |u|\},\] the set of all normal diwords on $X$ which is a linear basis of free dialgebra $D\langle X \rangle$, and $|u|$ is the number of letters in $u$. Therefore, any polynomial $f \in D \langle X \rangle$ has the form
\[ f = \alpha[\bar{f}]_{n} + \sum_{[u]_{m} \in [X^{+}]} \alpha_i [u_i]_{m}, \]
where $[\bar{f}]_{n}$ and $[u_i]_{m}$ are normal diwords in $X$, $[\bar{f}]_{n} > [u_i]_{m}$. Polynomial $f$ is called left (right) normed if the position of the center letter is in the most right(left) side of the normal form.
For any $h=[u]_{m} \in [X^{+}]$, we call $u$ the associative word of $h$, and $m$, the position of center of $h$, is denoted by $p(h)$. For example, if $u \colon= x_{1}x_{2} \dots x_{n} \in X^{+}$, $x_{t} \in X$, $h=[u]_{m}$, $ 1\leq m \leq n$, then $p(h)=m$, and with the notation as in \cite{3},
$ [u]_{m} \colon= x_{1} \vdash \dots \vdash x_{{m-1}} \vdash x_{m} \dashv x_{m+1} \dashv \dots \dashv x_{n}$.


Let $X$ be a totally ordered set. The monomial ordering according to the Bokut \textit{et al.}'s approach (lexicographic-weight) $<$ on normal diwords is defined as follows.
\[
 [u] < [v]  \Leftrightarrow \wt([u]) <_{\texttt{lex}} \wt([v])\ (\text{lexicographically}),
\]
where $\wt([u])=(n+m+1, m, x_{-m},\dots,x_0,\dots,x_n)$ with $[u]=x_{-m} \cdots \dot{x}_0 \cdots x_n$.
 Let $S \subset D \langle X \rangle $ be a monic subset of polynomials such that $Id(S)$ is the ideal of $D\langle X \rangle$ generated by $S$. An $S$-diword is a diword in $X \cup S$ with only one occurrence of $s \in S$. 
 A normal diword $[u]_{m}$ is said to be $S$-irreducible if $[u]_m$ is not equal to the leading monomial of any normal $S$-diword. Let $Irr(S)$ be the set of all $S$-irreducible diwords. Consider the following statements:
\begin{itemize}
    \item[(i)] $S$ is a Groebner-Shirshov basis in $D\langle X \rangle$.
    \item[(ii)]$Irr(S)$ is a $k$-basis of $D \langle X | S \rangle = D\langle X \rangle / Id(S) $.
\end{itemize}
It is shown in \cite{3} that $(i) \Rightarrow (ii)$ but $(ii) \not \Rightarrow (i)$. The main difference between the above result with the new version of Composition-Diamond Lemma in \cite{12} is the ordering defined on $[X^{+}]$. In fact, Bokut et al., in \cite{3} considered a fixed ordering and special definition of composition trivial modulo $S$, whereas Zhang et al., in \cite{12} introduced \emph{monomial-center ordering} on $[X^{+}]$ which makes the two above conditions equivalent. In the sequel, we recall the new version of Composition-Diamond Lemma in conformity with \cite{12}. 

\begin{definition}
Let $>$ be a deg-lex ordering on $X^{+}$. The deg-lex-center ordering $>_{d}$ on $[X^{+}]_{\omega}$ is defined as follows. For any $[u]_{m} , [v]_{n} \in [X^{+}]_{\omega}$,
\[[u]_{m} >_{d} [v]_{n} ~ \text{if} ~ (u,m) > (v,n) ~ \text{lexicographically.} \] 
\end{definition}
For any nonzero polynomial $f \in D \langle X \rangle $, let us denote by $\bar{f}$ the leading monomial of $f$ with respect
to the ordering $>$, $lt(f)$ the leading term of $f$, $lc(f)$ the coefficient of $\bar{f}$ and $\tilde{f}$ the
associative word of $\bar{f}$. Polynomial $f$ is called monic if $lc(f) = 1$. A nonempty subset $S$ of $D\langle X \rangle $ is called monic if $s$ is monic for all $s \in S$.
\begin{definition}
A nonzero polynomial $f \in D \langle X \rangle$ is strong if $\tilde{f} > \tilde{r_{f}}$, where $r_{f}\colon= f-lt(f)$.
\end{definition}
An $S$-diword $g$ is a normal diword on $X \cup S$ with only one occurance of $s \in S$. If this is the case and 
\[ g=[x_{i_{1}} \dots x_{i_{k}} \dots x_{i_{n}}]_{m} {{\mid}_{{x_{i_{k}}} \mapsto s} },\]
where $ 1 \leq k \leq n$, $x_{i_{l}} \in X$, $1 \leq j \leq n$, then we also call $g$ an $s$-diword. 
\begin{definition}
An $S$-diword is called a normal $S$-diword if either $k=m$ or $s$ is strong.
\end{definition}
Let $(asb)$ be a normal $S$-diword. Then $\overline{(asb)}=[a\tilde{s}b]_{l}$ for some $l \in P([asb])$, where
\begin{align*}
    P([asb]) \colon= &\{ n \in \mathbb{Z}^{+} \mid 1 \leq n \leq |a| \} \cup \{|a|+ p(\bar{s}) \} \\ &\cup \{ n \in \mathbb{Z}^{+} \mid |a \tilde{s}| < n \leq |a\tilde{s}b|\},
\end{align*} if $s$ is strong, otherwise, 
\begin{align*}
    P([asb]) \colon= \{ |a| + p(\bar{s})\}.
\end{align*}
In the following, we recall the available compositions between monic polynomials in $D \langle X \rangle$. 
\begin{definition}
Let $f$ and $g$ be monic polynomials in $D \langle X \rangle$. 
\begin{itemize}
    \item [(i)] If $f$ is not strong, then $x \dashv f$ is called the composition of left multiplication of $f$ for all $x \in X$ and $f \vdash [u]_{|u|}$ is called the composition of right multiplication of $f$ for all $u \in X^{+}$.
    \item[(ii)] Suppose that $w=\tilde{f}=a\tilde{g}b$ for some $a,b \in X^{\ast}$ and $(agb)$ is a normal $g$-diword.
    \begin{itemize}
        \item [(a)] If $p(\bar{f}) \in P([agb])$, then the composition of inclusion of $f$ and $g$ is defined as
        \[(f,g)_{\bar{f}}=f-[agb]_{p(\bar{f})}.\]
        \item[(b)] If $p(\bar{f}) \notin P([agb])$ and both $f$ and $g$ are strong, then for any $x \in X$ the composition of left multiplication inclusion is defined as 
        \[  (f,g)_{[xw]_{1}}=[xf]_{1} - [xagb]_{1} \] 
        and 
        \[(f,g)_{[wx]_{|wx|}}=[fx]_{|wx|} - [agbx]_{|wx|}\]
        is called the right multiplicative inclusion of $f$ and $g$.
    \end{itemize}
    \item[(iii)] Suppose that there exists $w=\tilde{f}b=a\tilde{g}$ for some $a,b \in X^{\ast}$ such that $|{\tilde{f}}|+|\tilde{g}|>|w|$, $(fb)$ is a normal $f$-diword and $(ag)$ is a normal $g$-diword.
    \begin{itemize}
        \item[(a)] If $P([fb]) \cap P([ag]) \neq 0$, then for any $m \in P([fb])\cap P([ag])$ we call 
        \[ (f,g)_{[w]_{m}}=[fb]_{m} - [ag]_{m} \]
        the composition of intersection of $f$ and $g$.
        \item[(b)] If $P([fb]) \cap P([ag]) = 0$ and both $f$ and $g$ are strong, then for any $x \in X$ we call 
        \[ (f,g)_{[xw]_{1}}=[xfb]_{1} - [xag]_{1} \]
        the composition of left multiplicative intersection of $f$ and $g$, and 
        \[ (f,g)_{[wx]_{|wx|}} = [fbx]_{|wx|} - [agx]_{|wx|} \]
        the composition of right multiplicative intersection of $f$ and $g$. 
    \end{itemize}
\end{itemize}
\end{definition}
\subsection*{Triviality criteria}\label{triviality}
Let $S$ be a monic subset of $D \langle X \rangle$. A polynomial $h \in D \langle X \rangle$ is trivial modulo $S$ and denoted by 
\[ h  \equiv 0 \mod (S) \]
if $h=\sum \alpha_{i} [a_{i}s_{i}b_{i}]_{m_{i}}$, where $\alpha_{i} \in K$, $a_i,b_i \in X^{\ast}$, $s_{i} \in S$ and $\overline{[a_is_ib_i]_{m_i}} \leq \bar{h}$.
\\A monic set $S$ is called \emph{Groebner-Shirshov basis} in $D \langle X \rangle$ if any composition of polynomials in $S$ is trivial modulo $S$. The next theorem is the new version of Composition-Diamond Lemma (CD-Lemma) for the case of dialgebras with respect to monomial-center ordering and new triviality criteria.
\begin{theorem}\label{CDlemma}\cite{12}
Let $S$ be a monic subset of $D \langle X \rangle$, $>$ a deg-lex-center ordering on $[X^{+}]$ and $Id(S)$ the ideal of  $D \langle X \rangle$ generated by $S$. Then the following statements are equivalent:
\begin{itemize}
    \item[(i)] $S$ is a Groebner-Shirshov basis in $D \langle X \rangle$.
    \item[(ii)] $f \in Id(S)$ implies $\bar{f}=\overline{[asb]}_{{m}}$ for some normal $S$-diword $[asb]_{m}.$
    \item[(iii)] $Irr(S)=\{[u]_n \in [X^{+}] \mid [u]_{n} \neq \overline{[asb]}_{m}$ for any normal $S$-diword $[asb]_{m}$ \} is a $K$-basis of $D \langle X | S \rangle = D \langle X \rangle / Id(S) $.
\end{itemize}
\end{theorem}

\section{HNN-extension of dialgebras}

\begin{definition}\label{derivation}
For a dialgebra $D$, a derivation is a map $d \colon D \rightarrow D$, which is  linear and  satisfies:
$
d(x \lp y)= d(x) \lp y + x \lp d(y)$ and $   d(x\rp y)= d(x) \rp y + x \rp d(y),
$
for all $x,y \in D$. 
\end{definition}

Let $ D$ be a dialgebra and $A$ be a subalgebra of $D$. Let $d \colon A \to D$ be a derivation defined on the subalgebra $A$. Then the corresponding HNN-extension is defined as 
\begin{equation}\label{HNN-extension}
   D_{d}^{\ast}= \langle D,t \mid a \dashv t - t \vdash a =d(a), ~ a \in A \rangle. 
\end{equation}
Here $t$ is a new symbol not belonging to $D$. 
Let assume that $X^{\prime} = X \cup \{t \}$, where $X$ is a well-ordered basis of $D$ and $t< X$. Let also denote by $Y$ the basis of $A$. We consider the following polynomials
\begin{itemize}
    \item $f(x,y)=[xy]_1 - \sum_{v} \alpha_{xy}^{v} v$ 
    \item $g(x,y)=[xy]_2 - \sum _{v} \beta_{xy}^{v} v$
    \item $h_{z}=[zt]_{1} - [tz]_{2} - \sum_{v} \delta _{z}^{v} v,$
\end{itemize}
where $v$ is an arbitrary element in $X$,
and $x, y\in X$ such that $x>y$ and $z\in Y$. Let us consider $D \langle X^{\prime} \mid S \rangle$ as a presentation of HNN-extension $D_{d}^{\ast}$ (\ref{HNN-extension}) through structural constants, where $S=\{f,g,h_{z}\}$.
We consider $x\dashv y=\sum_v\alpha_{xy}^vv$, for some $\alpha_{xy}^v\in \mathbb{K}$.
Similarly, we have $x\vdash y=\sum_v\beta_{xy}^vv$, for some scalars $\beta_{xy}^v$. Note that these scalars satisfy some relations according to the associativity laws; i.e.  we have
\begin{equation}\label{eq1}
    \sum_{v, u} \beta_{yz}^v \alpha_{xv}^u = \sum_{v, u} \alpha_{yz}^v \alpha_{xv}^u,
\end{equation}
\begin{equation}\label{eq2}
   \sum_{v, u} \alpha_{xy}^v \beta_{vz}^u = \sum_{v, u} \beta_{yz}^v \beta_{xv}^u, 
\end{equation}
\begin{equation}\label{eq3}
  \sum_{v, u} \alpha_{yz}^v \beta_{xv}^u = \sum_{v, u} \beta_{xy}^u \alpha_{yz}^v.
\end{equation}

Note  also that, since $A$ is a subalgebra, so for $x, y\in Y$ and $v\in X\setminus Y$, we have $\alpha_{xy}^v=\beta_{xy}^v=0$. Consider the derivation $d\colon A\to D$. For any $x\in Y$, there are scalars $\delta_x^v$ such that
\[
d(x)=\sum_{v}\delta_x^vv,
\]
and the Definition \ref{derivation} implies that  
\begin{equation}\label{derivationrelation}
    d(\sum \alpha_{xy}^{v} v)=\sum \delta_{x}^{v} v \dashv y + \sum \delta_{y}^{v} x \dashv v.
\end{equation}

\begin{proposition} Let $X = \{x_i | i \in I\}$ be a well-ordered set and $D_{d}^{\ast}\langle X \mid S\rangle $ be the presentation of HNN-extension of dialgebra $D$, where $S=\{f(x,y),g(x,y),h_z \}$ are the polynomials through structural constants. Then
\begin{itemize}
    \item [(i)] The relations $S=\{f,g,h \}$ form a Groebner-Shirshov basis for HNN-extension of dialgebras with respect to deg-lex-center ordering.
    \item[(ii)] The set 
    \begin{align*}
        &\{[x_{i_{1}} \dots x_{i_{n}}]_{1} \mid x_{i_{1}} \leq \dots \leq x_{i_{n-1}},~ x_{i_{l}}\in X, 1\leq l\leq n, ~n \in \mathbb{Z}^{+}\}\\
        & \cup \{ [x_{j_{1}} \dots x_{j_{m}}]_{m} \mid x_{j_{1}} \leq \dots \leq x_{j_{m}},~ x_{j_{k}}\in X, 1\leq k \leq m, ~m \in \mathbb{Z}^{+} \}\\
        &\cup \{[tx_{j_{1}} \dots x_{j_{m}}]_{1} \mid~ x_{j_{1}} \leq \dots \leq x_{j_{m}}, x_{j_{k}}\in X \} \\ &\cup \{[tx_{j_{1}} \dots x_{j_{m}}]_{m+1} \mid~ x_{j_{1}} \leq \dots \leq x_{j_{m}}, x_{j_{k}}\in X \}
    \end{align*} is a normal form of the elements $D_{d}^{\ast}\langle X \mid S\rangle $.
\end{itemize}

\end{proposition}

\begin{proof}
(i)   
We compute all possible compositions between elements of $S$. Let us denote by, for example, $ f \wedge g$  the composition of the polynomials $f$ and $g$.  We note that $f$, $g$ and $h$ are strong polynomial.
Let assume that $x>y>z$. We firstly check the intersection compositions. We put $f_1=[xy]_1 - \sum _{v} \alpha_{xy}^{v} v $ and $f_2= [yz]_1 - \sum _{v} \alpha_{yz}^{v} v,$ then for the intersection composition  $f_1 \wedge f_2$ we  have  $w=xyz$ and $P[f_{1}z] \cap P[xf_{2}]=\{1\}$. Therefore,
\begin{align*}
         (f_1,f_2)_{[xyz]_1}&= [xyz]_{1} -\sum _{v} \alpha_{xy}^{v} v\dashv z 
         -[xyz]_{1} + \sum _{v} \alpha_{yz}^{v} x \dashv v  \\
         &=\sum _{v} \alpha_{yz}^{v} x \dashv v -\sum _{v} \alpha_{xy}^{v} v\dashv z \\
         &= \sum_{u} \alpha_{yz}^{v} \alpha_{xv}^{u} u - \sum_{u} \alpha_{xy}^{v} \alpha_{vz}^{u} u \\
         &= \sum_{v, u} \beta_{yz}^v \alpha_{xv}^u u - \sum_{u} \alpha_{xy}^{v} \alpha_{vz}^{u} u =0.
\end{align*}
Also, we put $g_1=[xy]_2 - \sum _{v} \beta_{xy}^{v} v $ and $g_2= [yz]_2 - \sum _{v} \beta_{yz}^{v} v,$ then for the intersection composition  $g_1 \wedge g_2$ we  have  $w=xyz$ and $P[g_{1}z] \cap P[xg_{2}]=\{3\}$. Therefore,
\begin{align*}
         (g_1,g_2)_{[xyz]_3}&= [xyz]_{3} -\sum _{v} \beta_{xy}^{v} v\vdash z 
         -[xyz]_{3} + \sum _{v} \beta_{yz}^{v} x \vdash v  \\
         &=\sum _{v} \beta_{yz}^{v} x \vdash v -\sum _{v} \beta_{xy}^{v} v\vdash z \\
         &= \sum_{u} \beta_{yz}^{v} \beta_{xv}^{u} u - \sum_{u} \beta_{xy}^{v} \beta_{vz}^{u} u \\
         &= \sum_{v, u} \alpha_{xy}^v \beta_{vz}^u u - \sum_{u} \beta_{xy}^{v} \beta_{vz}^{u} u =0.
\end{align*}
Now we compute intersection composition of $f$ and $g$. Let us consider $f=[xy]_1 - \sum _{v} \alpha_{xy}^{v} v $ and $g=[yz]_2 -\sum_{v} \beta_{yz}^v v$, then we have $w=xyz$ and $P[g_{1}z] \cap P[xg_{2}]=\{1,3\}$. Therefore, 
\begin{align*}
         (f,g)_{[xyz]_1}&= [xyz]_{1} -\sum _{v} \alpha_{xy}^{v} v\dashv z 
         -[xyz]_{1} + \sum _{v} \beta_{yz}^{v} x \dashv v  \\
         &=\sum _{v} \beta_{yz}^{v} x \dashv v -\sum _{v} \alpha_{xy}^{v} v\dashv z \\
         &= \sum_{u} \beta_{yz}^{v} \alpha_{xv}^{u} u - \sum_{u} \alpha_{xy}^{v} \alpha_{vz}^{u} u \\
         &= \sum_{v, u} \alpha_{yz}^v \alpha_{xv}^u - \sum_{u} \alpha_{xy}^{v} \alpha_{vz}^{u} u  =0.
\end{align*}
The intersection composition $(f,g)_{[xyz]_3}$ is calculated similarly and it is also trivial modulo $S$. There is no intersection composition between $g$ and $f$. Let check the intersection composition of $f \wedge h$. We put $f=[xy]_1 - \sum _{v} \alpha_{xy}^{v} v $ and $h=[yt]_1 - [ty]_{2} - \sum_{v} \delta _{x}^{v} v, $ so for the intersection composition  $f \wedge h$ we  have  $w=xyt$ and $P[ft] \cap P[xh]=\{1\}$. Therefore,
\begin{align*}
         (f,h)_{[xyt]_1}&= [xyt]_{1}-\sum _{v} \alpha_{xy}^{v} v\dashv t -[xyt]_1  + [xty]_{1} + \sum _{v} \delta_{y}^{v} x \dashv v\\
         &=[xty]_{1} - \sum _{v} \alpha_{xy}^{v} v\dashv t + \sum _{v} \delta_{y}^{v} x \dashv v\\
         &=([xt]_{1} - [tx]_{2} -\sum_{v} \delta_{x}^{v} v) \dashv y + [tx]_{2} \dashv y \\&+ \sum_{v} \delta_{x}^{v} v \dashv y - \sum _{v} \alpha_{xy}^{v} v\dashv t + \sum _{v} \delta_{y}^{v} x \dashv v \\
         &= ([xt]_{1} - [tx]_{2} -\sum_{v} \delta_{x}^{v} v) \dashv y \quad \text{(by relation \ref{derivationrelation})}\\
         &=h_x \dashv y.
\end{align*}
We have $\overline{(f,h)_{[xyt]_1}}=[xt]_{1}\dashv y < [w]=[xyt]_{1}$. This shows that $(f, h)_{[xyt]_1}$ is trivial modulo $S$. There is no intersection composition $g \wedge h$ and $h_x \wedge h_y$. Therefore, $S=\{f,g,h\}$ is a Groebner-Shirshov basis for $D_{d}^{\ast}$.

(ii) This part follows from Theorem \ref{CDlemma}.
\end{proof}

\end{document}